\documentclass{article}

\usepackage{graphicx}
\usepackage{indentfirst}
\usepackage{amsmath,amsfonts,amsthm,amssymb}
\usepackage{mathrsfs}
\usepackage{amscd}

\def\co{\colon\thinspace}
\DeclareMathAlphabet{\mathsfsl}{OT1}{cmss}{m}{sl}

\newtheorem{thm}{Theorem}[section]
\newtheorem{lem}[thm]{Lemma}
\newtheorem{cor}[thm]{Corollary}
\newtheorem{prop}[thm]{Proposition}

\theoremstyle{definition}
\newtheorem{defn}[thm]{Definition}

\newtheorem{rem}[thm]{Remark}

\begin{document}

\title{Homological actions on sutured Floer homology}

\author{{\Large Yi NI}\\{\normalsize Department of Mathematics, Caltech, MC 253-37}\\
{\normalsize 1200 E California Blvd, Pasadena, CA
91125}\\{\small\it Emai\/l\/:\quad\rm yini@caltech.edu}}

\date{}
\maketitle

\begin{abstract}
We define the action of the homology group $H_1(M,\partial M)$ on the sutured Floer homology $SFH(M,\gamma)$.
It turns out that the contact invariant $EH(M,\gamma,\xi)$ is usually sent to zero by this action. This fact allows us to refine an earlier result proved by Ghiggini and the author. 
As a corollary, we classify knots in $\#^n(S^1\times S^2)$ which have simple knot Floer homology groups: They are essentially the Borromean knots.
\end{abstract}

\section{Introduction}

Heegaard Floer homology, introduced by Ozsv\'ath and Szab\'o
\cite{OSzAnn1}, is a powerful theory in low-dimensional topology.
In its most fundamental form, the theory constructs a chain
complex $\widehat{CF}(Y)$ for each closed oriented $3$--manifold
$Y$, such that the homology $\widehat{HF}(Y)$ of the chain complex
is a topological invariant of $Y$. In \cite{Ju1}, Juh\'asz adapted the construction of $\widehat{HF}(Y)$ to sutured manifolds, hence defined the sutured Floer homology $SFH(M,\gamma)$ for a balanced sutured manifold $(M,\gamma)$.

In \cite{OSzAnn1}, there is an action $A_{\zeta}$ on the Heegaard Floer homology for any  $\zeta\in H_1(Y;\mathbb Z)/\mathrm{Tors}$, which satisfies that $A^2_{\zeta}=0$. 
This induces a $\Lambda^*(H_1(Y;\mathbb Z)/\mathrm{Tors})$--module structure on the Heegaard Floer homology.
The goal of this paper is to define an action $A_{\zeta}$ on $SFH(M,\gamma)$ for any $\zeta\in H_1(M,\partial M;\mathbb Z)/\mathrm{Tors}$, 
hence make $SFH(M,\gamma)$ a $\Lambda^*(H_1(M,\partial M;\mathbb Z)/\mathrm{Tors})$--module. 

Juh\'asz \cite{Ju2} proved that if there is a sutured manifold decomposition
$$(M,\gamma)\stackrel{S}{\rightsquigarrow}(M',\gamma')
$$
with good properties, then ${SFH}(M',\gamma')$ is a direct summand of ${SFH}(M,\gamma)$. Hence there is an inclusion map
$\iota\co SFH(M',\gamma')\to SFH(M,\gamma)$ and a projection map
$\pi\co SFH(M,\gamma)\to SFH(M',\gamma')$ such that
$$\pi\circ\iota=\mathrm{id}.$$
It is natural to expect that Juh\'asz's decomposition formula respects the action of $H_1(M,\partial M;\mathbb Z)/\mathrm{Tors}$. Our 
main theorem confirms this expectation.

\begin{thm}\label{thm:Decomp}
Let $(M,\gamma)$ be a balanced sutured manifold and let
$$(M,\gamma)\stackrel{S}{\rightsquigarrow}(M',\gamma')
$$
be a well-groomed sutured manifold decomposition. Let $$i_*\co H_1(M,\partial M)\to H_1(M,(\partial M)\cup S)\cong H_1(M',\partial M')$$
be the map induced by the inclusion map $i\co (M,\partial M)\to (M,(\partial M)\cup S)$. If $\zeta\in H_1(M,\partial M)$, then $i_*(\zeta)\in H_1(M',\partial M')$. Then 
$$
\iota\circ A_{i_*(\zeta)}=A_{\zeta}\circ\iota,\quad A_{i_*(\zeta)}\circ\pi=\pi\circ A_{\zeta}.
$$
where $\iota,\pi$ are the inclusion and projection maps defined before.
\end{thm}

A corollary of Theorem~\ref{thm:Decomp} is the following one.

\begin{cor}\label{cor:KerFiber}
Suppose $K$ is a nullhomologous knot in a closed oriented manifold $Y$ such that $Y-K$ is irreducible. Suppose $F$ is a Thurston norm minimizing Seifert surface for $K$. Let 
$$\mathrm{Ker} A=\left\{x\in\widehat{HFK}(Y,K,[F],-g)\Big|\:A_{\zeta}(x)=0 \text{ for all }\zeta\in H_1(Y)/\mathrm{Tors}\right\},$$
which is a subgroup of $\widehat{HFK}(Y,K,[F],-g)$. If $F$ is not the fiber of any fibration (if there is any) of $Y-K$, then the rank of $\mathrm{Ker} A$ is at least $2$.
\end{cor}

Ghiggini \cite{Gh} and the author \cite{NiFibred} have proved that if $F$ is not the fiber of any fibration of $Y-K$, then the rank of $\widehat{HFK}(Y,K,[F],-g)$ is at least $2$. Since $\mathrm{Ker} A$ is a subgroup of $\widehat{HFK}(Y,K,[F],-g)$, 
the above corollary can be viewed as  a refinement of the theorem of Ghiggini and the author.

\subsection{Knots in $\#^n(S^1\times S^2)$ with simple knot Floer homology}

Corollary~\ref{cor:KerFiber} is most useful when $\widehat{HF}(Y)$ has a rich $\Lambda^*(H_1(Y)/\mathrm{Tors})$--module structure. 
As an illustration, we will study knots in $\#^n(S^1\times S^2)$ that have simple knot Floer homology.

Suppose $K$ is a rationally null-homologous knot in $Y$,
Ozsv\'ath--Szab\'o \cite{OSzKnot,OSzRatSurg} and Rasmussen
\cite{RasThesis} showed that $K$ specifies a filtration on
$\widehat{CF}(Y)$. The associated homology of the filtered chain
complex is the knot Floer homology $\widehat{HFK}(Y,K)$.
From the construction of knot Floer homology, one sees that
$$\mathrm{rank}\:\widehat{HFK}(Y,K)\ge\mathrm{rank}\:\widehat{HF}(Y),$$
for any rationally null-homologous knot $K\subset Y$. When the
equality holds, we say that the knot has {\it simple} knot Floer
homology.

To an oriented null-homologous $n$--component link $L\subset Y$,
Ozsv\'ath and Szab\'o \cite{OSzKnot} associated a null-homologous knot
$\kappa(L)\subset \kappa(Y)=Y \#^{n-1}(S^1\times S^2)$, and
defined the knot Floer homology of $L$ to be the knot Floer
homology of $\kappa(L)$. Hence
$$\mathrm{rank}\:\widehat{HFK}(Y,L)\ge\mathrm{rank}\:\widehat{HF}(\kappa(Y))=2^{n-1}\mathrm{rank}\:\widehat{HF}(Y).$$
When the equality holds, we say that the link has {\it simple}
knot Floer homology.

Clearly, the unknot in $Y$ always has simple Floer homology. Sometimes there are other knots with this property. For example, 
the core of a solid torus in the genus--$1$ Heegaard splitting of a lens space has simple Floer homology. Moreover, if two knots $(Y_1,K_1)$ and $(Y_2,K_2)$ have simple 
Floer homology groups, then their connected sum $(Y_1\#Y_2,K_1\#K_2)$ also has simple Floer homology. In particular, $(Y_1\#Y_2,K_1)$ has simple Floer homology.

It is an interesting problem to determine all knots with simple Floer homology.
For example, Hedden \cite{HedBerge} and Rasmussen \cite{RasBerge}
showed that if a knot $L\subset S^3$ admits an integral lens space
surgery, then the core of the surgery is a knot in
the lens space with simple knot Floer homology group. Hence the classification of knots with simple Floer homology groups in
lens spaces will lead to a resolution of Berge's conjecture on
lens space surgery.

For certain $3$--manifolds, the classification of 
knots with simple Floer homology groups are already known. A deep theorem of Ozsv\'ath--Szab\'o
\cite[Theorem~1.2]{OSzGenus} implies that the only 
knot in $S^3$ with simple Floer homology group is the unknot. The author \cite{NiLink} generalized
Ozsv\'ath--Szab\'o's theorem. As a corollary, if $Y$ is a integer
homology sphere which is an $L$--space, then the unknot is the
only knot in $Y$ with simple Floer homology group. Eftekhary \cite{Ef} has announced a
generalization of this result to knots in any homology sphere.

\begin{figure}
\begin{picture}(340,85)
\put(70,0){\scalebox{0.5}{\includegraphics*[0pt,0pt][400pt,
170pt]{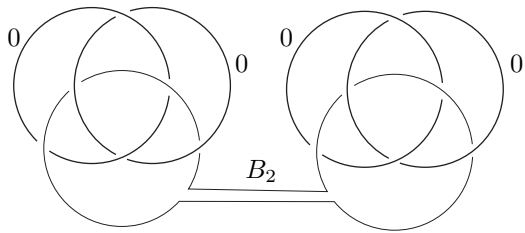}}}

\put(258,60){$0$}

\put(154,60){$0$}

\put(171,70){$0$}

\put(68,70){$0$}

\put(158,20){$B_2$}

\end{picture}
\caption{The Borromean knot $B_2$}\label{fig:Borro}
\end{figure}

We will determine all knots in $Y_n=\#^n(S^1\times
S^2)$ with simple knot Floer homology. Besides the unknot, there is a class of knots called Borromean knots which have simple Floer homology groups.
Consider the Borromean rings. Perform $0$--surgery on two components of the Borromean rings, then the third component becomes a knot in $Y_2$, called the Borromean knot $B_1$.
The Borromean knot $B_k\subset Y_{2k}$ is obtained by taking the connected sum of $k$ $B_1$'s, and $B_0$ is understood to be the unknot in $S^3$. See Figure~\ref{fig:Borro}. Borromean knots are characterized by the fact that $B_k$ is the binding of an open book decomposition of $Y_{2k}$, 
such that the page is a genus $k$ surface with one boundary component, and the monodromy is the identity map.
It is known that $\widehat{HFK}(Y_{2k},B_k)$ has rank $2^{2k}$ \cite{OSzKnot}.

\begin{thm}\label{thm:Borromean}
Suppose $K\subset Y_n=\#^n(S^1\times S^2)$ is a null-homologous
knot with simple knot Floer homology. Namely,
$$\mathrm{rank}\:\widehat{HFK}(Y_n,K)=2^{n-1}.$$
Then there exists a non-negative integer $k\le\frac n2$ such that
$(Y_n,K)$ is obtained from the Borromean knot $(Y_{2k},B_k)$ by
taking connected sum with $Y_{n-2k}$ in the complement of $B_k$.
\end{thm}

A difference between Theorem~\ref{thm:Borromean} and the previously known classification results of simple knots is, the simple knots in $Y_n$ include some nontrivial knots. 
Such situation also appears when one tries to classify simple knots in lens spaces.

For links in $S^3$, it is easy to determine which ones have simple Floer
homology. Our result is as follows.

\begin{prop}\label{prop:LinkRank}
Suppose $L$ is an $n$--component link in $S^3$. If the rank of its
knot Floer homology $\widehat{HFK}(L)$ is $2^{n-1}$, then $L$ is
the $n$--component unlink.
\end{prop}

\begin{rem}
It is proved by the author \cite{NiLink} that knot Floer
homology detects unlinks. However, the main result there does not imply that the rank of knot Floer homology
detects unlinks. Proposition~\ref{prop:LinkRank} remedies this omission.
\end{rem}

This paper is organized as follows. In Section~2 we will define the action on sutured Floer homology and prove Theorem~\ref{thm:Decomp}. In Section~3 we will study the effect of the action on the 
contact invariant $EH(M,\gamma,\xi)$. Then we will prove Corollary~\ref{cor:KerFiber}. In Section~4 we will use Heegaard Floer homology and combinatorial group theory to determine all knots in $Y_n$ that have simple 
Floer homology. In Section~5 we will show the unlinks are
the only links in $S^3$ that have simple Floer homology.

\

\noindent{\bf Acknowledgements.}\quad This work was carried out
when the author participated the ``Homology Theories of Knots and
Links'' program at MSRI and when the author visited Princeton University. The author wishes to thank MSRI and David Gabai for their hospitality. The author was
partially supported by an AIM Five-Year Fellowship and NSF grant
number DMS-1021956.

\section{Sutured Floer homology and the homological action}

In this section, we will define the homological action on the sutured Floer homology and study its behavior under sutured manifold decompositions. 
We will assume the readers have some familiarity with sutured manifold theory and sutured Floer homology.

\subsection{The definition of the action}

\begin{defn}
A {\it sutured manifold} $(M,\gamma)$ is a compact oriented
3--manifold $M$ together with a set $\gamma\subset \partial M$ of
pairwise disjoint annuli. The
core of each component of $\gamma$ is a {\it suture}, and the
set of sutures is denoted by $s(\gamma)$.

Every component of $R(\gamma)=\partial M-\mathrm{int}(\gamma)$ is
oriented. Define $R_+(\gamma)$ (or $R_-(\gamma)$) to be the union
of those components of $R(\gamma)$ whose normal vectors point out
of (or into) $M$. The orientations on $R(\gamma)$ must be coherent
with respect to $s(\gamma)$.

A {\it balanced sutured manifold} is a sutured manifold $(M,\gamma)$ with $\chi(R_+(\gamma))=\chi(R_-(\gamma))$,
and $\gamma$ intersects each component of $\partial M$.
\end{defn}

Let $(M,\gamma)$ be a sutured manifold, and $S$ a properly
embedded surface in M. According to Gabai \cite{G1}, there is a natural way to put a sutured manifold structure on 
$M'=M\backslash\nu(S)$. This process is
a {\it sutured manifold decomposition}
$$(M,\gamma)\stackrel{S}{\rightsquigarrow}(M',\gamma').$$

\begin{defn}
A sutured manifold decomposition is {\it well-groomed}, if
for every component $V$ of $R(\gamma)$, $V\cap S$ is a union of parallel oriented
nonseparating simple closed curves or arcs.
\end{defn}

\begin{defn}
A {\it sutured manifold hierarchy} is a sequence of decompositions
$$(M_0,\gamma_0)\stackrel{S_0}{\rightsquigarrow}(M_1,\gamma_1)\stackrel{S_1}{\rightsquigarrow}(M_2,\gamma_2)\stackrel{S_2}{\rightsquigarrow}\cdots\stackrel{S_{n-1}}{\rightsquigarrow}(M_n,\gamma_{n}),
$$
such that $(M_n,\gamma_n)$ is a disjoint union of $(D^2\times I,(\partial D^2)\times I)$'s.
\end{defn}

A fundamental theorem of Gabai \cite{G1} says that for any taut sutured manifold, there exists a well-groomed sutured manifold hierarchy.

Suppose
$(\Sigma,\mbox{\boldmath$\alpha$},\mbox{\boldmath$\beta$})$
is a sutured Heegaard diagram for a balanced sutured manifold
$(M,\gamma)$. Let $\omega$ be a relative $1$--cycle on $\Sigma$,
such that it is in general position with the $\alpha$-- and
$\beta$--curves. Namely, $\omega=\sum k_ic_i$, where
$k_i\in\mathbb Z$, each $c_i$ is a properly immersed oriented
curve on $\Sigma$, such that $c_i$ is transverse to $\alpha$-- and $\beta$--curves,
and $c_i$ does not contain any intersection point of $\alpha$--
and $\beta$--curves.

We can also regard $\omega$ as a relative $1$--cycle representing
a class in $H_1(M,\partial M)$. On the other hand, any homology
class in $H_1(M,\partial M)$ can be represented by a
relative $1$--cycle on $\Sigma$.

If
$\phi$ is a topological Whitney disk connecting $\mathbf x$ to
$\mathbf y$, let $\partial_{\alpha}\phi=(\partial\phi)\cap\mathbb
T_{\alpha}$. We can also regard $\partial_{\alpha}\phi$ as a
multi-arc that lies on $\Sigma$ and connects $\mathbf x$ to
$\mathbf y$. Similarly, we define $\partial_{\beta}\phi$ as a
multi-arc connecting $\mathbf y$ to $\mathbf x$. We define
$$a(\omega,\phi)=\#\widehat{\mathcal M}(\phi)\:\big(\omega\cdot(\partial_{\alpha}\phi)\big),$$
where $\omega\cdot(\partial_{\alpha}\phi)$ is the algebraic
intersection number of $\omega$ and $\partial_{\alpha}\phi$. Let
$$A_{\omega}\co SFC(M,\gamma)\to SFC(M,\gamma)$$ be
defined as
$$A_{\omega}(\mathbf x)=\sum_{\mathbf y\in\mathbb T_{\alpha}\cap\mathbb T_{\beta}}
\sum_{\{\phi\in\pi_2(\mathbf x,\mathbf
y)|\mu(\phi)=1\}}a(\omega,\phi)\mathbf
y.$$
As in \cite[Lemma~4.18]{OSzAnn1}, $A_{\omega}$ is a chain map. The following lemma shows that it induces a well defined action of $H_1(M,\partial M;\mathbb Z)/\mathrm{Tors}$ on $SFH(M,\gamma)$.

\begin{lem}\label{lem:ActHomo}
Suppose $\omega_1,\omega_2\subset \Sigma$ are two relative $1$--cycles which are
homologous in $H_1(M,\partial M;\mathbb Z)/\mathrm{Tors}$, then
$A_{\omega_1}$ is chain homotopic to $A_{\omega_2}$.
\end{lem}
\begin{proof}
Since $\omega_1$ and $\omega_2$ are homologous in $H_1(M,\partial M;\mathbb Z
)/\mathrm{Tors}$, there exists a nonzero integer $m$ such that
$m[\omega_1]=m[\omega_2]\in H_1(M,\partial M;\mathbb Z)$. 

\noindent{\bf Claim.} There exists a relative $2$--chain $B$ in $(\Sigma,\partial\Sigma)$, such that
$(\partial B)\backslash(\partial \Sigma)$ consists of $m\omega_2$, $m(-\omega_1)$, copies of
$\alpha$--curves and $\beta$--curves, and proper curves $\xi,\eta\subset \Sigma$ such that $\xi$ is 
disjoint from $\alpha$--curves and $\eta$ is disjoint from $\beta$--curves.

Consider the triple $(M,\partial M,\gamma)$, we get an exact sequence
$$H_1(\partial M,\gamma)\to H_1(M,\gamma)\to H_1(M,\partial M)\to 0.$$
As a consequence, if $m[\omega_1]=m[\omega_2]\in H_1(M,\partial M)$, then there exists an element $c\in H_1(\partial M,\gamma)\cong H_1(R(\gamma),\partial R(\gamma))$, 
such that $c+m[\omega_2]-m[\omega_1]$ is homologous to zero in $H_1(M,\gamma)$. Let $\xi'\subset R_-(\gamma),\eta'\subset R_+(\gamma)$ be proper curves  
such that $\xi'+\eta'$ represents $c\in H_1(R(\gamma),\partial R(\gamma))$. Using the gradient flow of a Morse function associated with the sutured diagram, we can project $\xi',\eta'$ to 
proper curves $\xi,\eta\subset\Sigma$ such that $\xi$ is disjoint from $\alpha$--curves and $\eta$ is disjoint from $\beta$--curves. Then $\xi+\eta+m[\omega_2]-m[\omega_1]$ is homologous to zero in $H_1(M,\gamma)$.
Using the fact that
$$H_1(M,\gamma)\cong H_1(\Sigma,\partial\Sigma)/([\alpha_1]\dots,[\alpha_g],[\beta_1],\dots,[\beta_g]),$$
we conclude that there is a relative $2$--chain $B$ in $(\Sigma,\partial\Sigma)$, such that
$(\partial B)\backslash(\partial \Sigma)$ consists of $m\omega_2$, $m(-\omega_1)$, $\xi,\eta$, and copies of
$\alpha$--curves and $\beta$--curves. This finishes the proof of the claim.

Perturbing $B$ slightly, we get a
$2$--chain $B'$ such that
$$(\partial B')\backslash(\partial\Sigma)=m\omega_2-m\omega_1+\sum (a_i\alpha'_i+b_i\beta'_i)+\xi+\eta,$$
where $\alpha_i',\beta_i'$ are parallel copies of $\alpha_i,
\beta_i$.

Let $\phi$ be a topological Whitney disk connecting $\mathbf x$ to
$\mathbf y$. Since $\alpha_i',\xi$ are disjoint from all $\alpha$--curves, we have $\alpha_i'\cdot\partial_{\alpha}\phi=\xi\cdot\partial_{\alpha}\phi=0$.
Similarly,
$$\beta_i'\cdot\partial_{\alpha}\phi=-\beta_i'\cdot\partial_{\beta}\phi=0, \quad\eta\cdot\partial_{\alpha}\phi=0.$$
We have \begin{equation}\label{eq:Diffn} n_{\mathbf
x}(B')-n_{\mathbf y}(B')=-((\partial
B')\backslash(\partial\Sigma))\cdot\partial_{\alpha}\phi=m(\omega_1-\omega_2)\cdot\partial_{\alpha}\phi\in
m\mathbb Z.
\end{equation}

Pick an intersection point $\mathbf x_0$ representing a relative Spin$^c$
structure $\mathfrak s$. After adding copies of $\Sigma$ to $B'$,
we can assume that $n_{\mathbf x_0}(B')$ is divisible by $m$.
Since any two intersection points representing $\mathfrak s$ are
connected by a topological Whitney disk, (\ref{eq:Diffn}) implies
that $n_{\mathbf x}(B')$ is divisible by $m$ for any $\mathbf x$
representing $\mathfrak s$.

Now we define a map $H\co SFC(M,\gamma,\mathfrak s)\to
SFC(M,\gamma,\mathfrak s)$ by letting
$$H(\mathbf x)=\frac{n_{\mathbf x}(B')}m\mathbf x.$$
It follows from (\ref{eq:Diffn}) that
$$A_{\omega_1}-A_{\omega_2}=\partial\circ H-H\circ\partial.$$
Namely, $A_{\omega_1},A_{\omega_2}$ are chain homotopic.
\end{proof}

Now the same argument as in \cite[Lemma~4.17]{OSzAnn1} shows that $A_{\zeta}$ is a differential for any $\zeta\in H_1(M,\partial M)/\mathrm{Tors}$, hence $A$ gives rise to an action of 
$\Lambda^*(H_1(M,\partial M)/\mathrm{Tors})$ on $SFH(M,\gamma)$.

\subsection{Sutured manifold decomposition}

Suppose
$$(M,\gamma)\stackrel{S}{\rightsquigarrow}(M',
\gamma')
$$
is a well-groomed sutured manifold decomposition.
In \cite{Ju2}, Juh\'asz constructed a surface diagram 
$$(\Sigma,\mbox{\boldmath${\alpha}$},
\mbox{\boldmath$\beta$},P)$$
adapted to $S$, where $P\subset\Sigma$ is a compact surface with corner such that $P\cap (\partial\Sigma)$ consists of exactly the vertices of $P$.
Moreover, $$\partial P=A\cup B,$$ where $A,B$ are unions of edges of $P$ with $A\cap B=P\cap(\partial \Sigma)$, $A\cap\mbox{\boldmath$\beta$}=\emptyset,
B\cap\mbox{\boldmath$\alpha$}=\emptyset$. $(\Sigma,\mbox{\boldmath${\alpha}$},
\mbox{\boldmath$\beta$})$ is a balanced diagram for $(M,\gamma)$. A balanced diagram 
$$(\Sigma',\mbox{\boldmath${\alpha}$}',
\mbox{\boldmath$\beta$}')$$
for $(M',\gamma')$ can be constructed as follows:
$$\Sigma'=(\Sigma\backslash P)\cup P_A\cup P_B,$$
where $P_A,P_B$ are two copies of $P$, and $\Sigma\backslash P$ is glued to $P_A$ along $A$ while glued to $P_B$ along $B$. There is a natural projection map $p\co \Sigma'\to \Sigma$, and
$$\mbox{\boldmath${\alpha}$}'=p^{-1}(\mbox{\boldmath${\alpha}$}),\mbox{\boldmath${\beta}$}'=p^{-1}(\mbox{\boldmath${\beta}$}).$$

The decomposing surface $S$ can be seen from the surface diagram as follows: $M$ can be got from $\Sigma\times[0,1]$ by adding $2$--handles along $\alpha_i\times0$'s and $\beta_j\times1$'s. 
Then $S\subset M$ is isotopic to the surface \begin{equation}\label{eq:PtoS}(P\times\frac12)\cup(A\times[\frac12,1])\cup(B\times[0,\frac12]).
\end{equation}

Let $O_P$ be the set of intersection points in $\mathbb T_{\alpha}\cap\mathbb T_{\beta}$ that are supported outside of $P$. 
Then $O_P$ consists of the points whose associated relative Spin$^c$ structures are ``extremal" with respect to $S$. And $O_P$ is in one-to-one correspondence with $\mathbb T_{\alpha'}\cap\mathbb T_{\beta'}$.

Using techniques introduced by Sarkar and Wang \cite{SW}, Juh\'asz proved that the surface diagram can be made ``nice". In particular, if $\phi$ is a holomorphic disk connecting two points in $O_P$ with $\mu(\phi)=1$, 
then the domain of $\phi$ is either an embedded bigon or square. Moreover, the following fact was contained in the proof of \cite[Proposition~7.6]{Ju2}.

\begin{lem}\label{lem:DiskP}
Suppose $\mathcal D$ is the domain of a holomorphic disk connecting two points in $O_P$ with $\mu(\phi)=1$, and $C$ is a component of $\mathcal D\cap P$. Then $C$ is either a bigon or a square. 
If $C$ is a bigon, then $C$ has either an $\alpha$--edge and an $A$--edge, or a $\beta$--edge and a $B$--edge. If $C$ is a square, then $C$ has either two opposite $\alpha$--edges and two opposite $A$--edges, 
or two opposite $\beta$--edges and two opposite $B$--edges.\qed
\end{lem}

\begin{figure}
\begin{picture}(340,68)
\put(30,0){\scalebox{0.8}{\includegraphics*[0pt,0pt][350pt,
85pt]{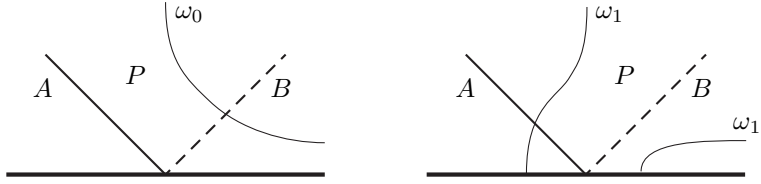}}}

\put(253,60){$\omega_1$}

\put(305,18){$\omega_1$}

\put(94,60){$\omega_0$}

\put(130,30){$B$}

\put(289,30){$B$}

\put(200,30){$A$}

\put(40,30){$A$}

\put(75,35){$P$}

\put(260,35){$P$}

\end{picture}
\caption{Changing $\omega_0$ near a corner of $P$}\label{fig:IntersectA}
\end{figure}

\begin{lem}\label{lem:IntersectA}
Suppose $\zeta\in H_1(M,\partial M)$, then $\zeta$ can be represented by a relative $1$--cycle $\omega\subset \Sigma$, such that $\omega$ intersects $\partial P$ in the interior of $A$. 
As a result, $\omega$ can be lifted to a relative $1$--cycle $\omega'\subset((\Sigma\backslash P)\cup_{A}P_A)\subset\Sigma'$, such that $p$ maps $\omega'$ homeomorphically to $\omega$, 
and $\omega'$ represents $i_*(\zeta)\in H_1(M',\partial M')$.
\end{lem}
\begin{proof}
Let $\omega_0\subset \Sigma$ be a relative $1$--cycle representing $\zeta$, such that $\omega_0$ intersects $\partial P$ transversely in the interior of the edges. 
Suppose $\omega_0$ has an intersection point with $B$. After homotoping $\omega_0$, we may assume this intersection point is near a corner of  $P$. As in Figure~\ref{fig:IntersectA}, 
we can replace $\omega_0$ by a new relative $1$--cycle $\omega_1$, such that $[\omega_1]=[\omega_0]\in H_1(\Sigma,\partial\Sigma)$, and $\#(\omega_1\cap B)=\#(\omega_0\cap B)-1$. 
Continuing this process, we get a relative $1$--cycle $\omega$ representing $\zeta$, which does not intersect $B$.

Cutting $\Sigma$ open along $B$, we get a surface homeomorphic to $(\Sigma\backslash P)\cup_{A}P_A$. Since $\omega$ does not intersect $B$, it lies in the new surface. Hence there is a corresponding relative
$1$--cycle $\omega'\subset(\Sigma\backslash P)\cup_{A}P_A$, such that $p$ maps $\omega'$ homeomorphically to $\omega$. 

$S$ is isotopic to a surface (\ref{eq:PtoS}). Since $\omega$ is disjoint from $B$, $\omega\times(\frac12-\epsilon)$ is disjoint from $S$. Cutting $M$ open along $S$, 
$\omega\times(\frac12-\epsilon)$ becomes the relative $1$--cycle $\omega'\times(\frac12-\epsilon)$. Hence $\omega'$ represents $i_*(\zeta)$ in $H_1(M',\partial M')$.
\end{proof}

Now we are ready to prove Theorem~\ref{thm:Decomp}

\begin{proof}[Proof of Theorem~\ref{thm:Decomp}]
The projection map $p\co \Sigma'\to\Sigma$ induces a bijection $$p_*\co\mathbb T_{\alpha'}\cap\mathbb T_{\beta'}\to O_P,$$ which then induces the inclusion map
$$SFC(M',\gamma')\to SFC(M,\gamma).$$
As argued in \cite{Ju2}, $p$ induces a one-to-one correspondence between the holomorphic disks for $SFC(M',\gamma')$ and the holomorphic disks for the chain complex generated by $O_P$. Let $\omega$ be the curve obtained in Lemma~\ref{lem:IntersectA}.
Suppose $\phi$ is a holomorphic disk connecting two points in $O_P$, $\phi'$ is the corresponding holomorphic disk connecting two points in $\mathbb T_{\alpha'}\cap\mathbb T_{\beta'}$. 
The intersection points of $\partial_{\alpha}\phi$ and $\omega$ outside of $P$ are in one-to-one correspondence with the  intersection points of  $\partial_{\alpha}\phi'$ and $\omega'$ outside of $P_A\cup P_B$.  

Let $C$ be a component of $\mathcal D\cap P$, where $\mathcal D$ is the domain of $\phi$. Let $C'$ be the corresponding component of $\mathcal D'\cap(P_A\cup P_B)$.
By Lemma~\ref{lem:DiskP}, $C\cap(\partial P)$ is contained in either $A$ or $B$. If $(C\cap(\partial P))\subset A$, then $C'\subset P_A$, and $(\partial_{\alpha}C)\cap\omega$ is in one-to-one correspondence 
with $(\partial_{\alpha}C')\cap\omega'$. If $(C\cap(\partial P))\subset B$, then $C'\subset P_B$, so $(\partial_{\alpha}C')\cap\omega'=\emptyset$. By Lemma~\ref{lem:DiskP}, in this case $C$ has no $\alpha$--edge, 
so $(\partial_{\alpha}C)\cap\omega=\emptyset$. This shows that
$$\omega\cdot(\partial_{\alpha}\phi)=\omega'\cdot(\partial_{\alpha}\phi').$$
Now our desired result follows from the definition of the homological action.
\end{proof}

\subsection{The homological action on Knot Floer homology}

The material in this subsection is not used in this paper. However, it is helpful to have in mind the symmetry stated in Proposition~\ref{prop:Symm}.

Suppose $K$ is a null-homologous knot in a closed $3$--manifold
$Y$. Let $$(\Sigma,\mbox{\boldmath${\alpha}$},
\mbox{\boldmath$\beta$},w,z)$$ be a doubly pointed Heegaard
diagram for $(Y,K)$ which is induced from a marked Heegaard
diagram. Fix a Spin$^c$ structure $\mathfrak s$ on $Y$ and let
$\xi\in\underline{\mathrm{Spin}}^c(Y,K)$ be a relative Spin$^c$
structure which extends it. Let $CFK^{\infty}(Y,K,\xi)$ be an
abelian group freely generated by triples $[\mathbf x,i,j]$ with
$$\mathbf x\in\mathbb T_{\alpha}\cap\mathbb T_{\beta},\quad\mathfrak
s_w(\mathbf x)=\mathfrak s$$ and $$\underline{\mathfrak
s}_m(\mathbf x)+(i-j)PD[\mu]=\xi.$$ The chain complex is endowed
with the differential
$$\partial^{\infty}[\mathbf x,i,j]=\sum_{\mathbf y\in\mathbb T_{\alpha}\cap\mathbb T_{\beta}}
\sum_{\{\phi\in\pi_2(\mathbf x,\mathbf
y)|\mu(\phi)=1\}}\#(\widehat{\mathcal M}(\phi))[\mathbf
y,i-n_w(\phi),j-n_z(\phi)].$$ The homology of
$(CFK^{\infty}(Y,K,\xi),\partial^{\infty})$ is denoted
$HFK^{\infty}(Y,K,\xi)$.

Suppose $\omega$ is a $1$--cycle on $\Sigma$. Let $$A_{\omega}[\mathbf x,i,j]=\sum_{\mathbf y\in\mathbb T_{\alpha}\cap\mathbb T_{\beta}}
\sum_{\{\phi\in\pi_2(\mathbf x,\mathbf
y)|\mu(\phi)=1\}}a(\omega,\phi)[\mathbf
y,i-n_w(\phi),j-n_z(\phi)].$$
As in \cite{OSzAnn1} and the arguments before, $A_{\omega}$ induces an action
of $\Lambda^*H_1(Y;\mathbb Z)/\mathrm{Tors}$ on
$HFK^{\infty}(Y,K,\xi)$.

There is a $U$--action on $CFK^{\infty}(Y,K,\xi)$ given by
$$U[\mathbf x,i,j]=[\mathbf x,i-1,j-1].$$
 Let
$CFK^{-,*}(Y,K,\xi)$ be the subcomplex of $CFK^{\infty}(Y,K,\xi)$
generated by $[\mathbf x,i,j]$ with $i<0$, and let
$CFK^{+,*}(Y,K,\xi)$ be its quotient complex. Moreover, let
$CFK^{0,*}(Y,K,\xi)\subset CFK^{+,*}(Y,K,\xi)$ be the kernel of
the induced $U$--action. The grading $j$ gives a filtration on
$CFK^{0,*}(Y,K,\xi)$, the associated graded complex is denoted
$\widehat{CFK}(Y,K,\xi)$. There are induced actions of
$A_{\omega}$ on the above complexes, and the actions induce
differentials $A_{[\omega]}$ on the corresponding homology groups.

When $\mathfrak s$ is a torsion Spin$^c$ structure over $Y$, as in
Ozsv\'ath--Szab\'o \cite{OSzKnot} there is an absolute $\mathbb
Q$--grading on $CFK^{\infty}(Y,K,\xi)$ and the induced complexes.
Let $\widehat{HFK}_d(Y,K,\xi)$ be the summand of
$\widehat{HFK}(Y,K,\xi)$ at the absolute grading $d$.

\begin{prop}\label{prop:Symm}
Let $\mathfrak s$ be a torsion Spin$^c$ structure over $Y$, and
let $\xi\in\underline{\mathrm{Spin}^c}(Y,K)$ be a relative
Spin$^c$ structure which extends $\mathfrak s$. Let $\zeta$ be
a homology class in $H_1(Y;\mathbb Z)/\mathrm{Tors}$. Then there
is an isomorphism
$$f\co\widehat{HFK}_d(Y,K,\xi)\stackrel{\cong}{\longrightarrow}\widehat{HFK}_{d-2m}(Y,K,J\xi),$$
such that the following diagram is commutative: $$\begin{CD}
\widehat{HFK}_d(Y,K,\xi) @>f>> \widehat{HFK}_{d-2m}(Y,K,J\xi)\\
@VVA_{\zeta}V @VVA_{-\zeta}V\\
\widehat{HFK}_d(Y,K,\xi) @>f>> \widehat{HFK}_{d-2m}(Y,K,J\xi),
\end{CD}$$
 where $2m=\langle
c_1(\xi),[\widehat{F}]\rangle$ for any Seifert surface $F$ for
$K$.
\end{prop}
\begin{proof}
Let $$\Gamma_1=(\Sigma,\mbox{\boldmath${\alpha}$},
\mbox{\boldmath$\beta$},w,z)$$ be a doubly pointed Heegaard
diagram for $(Y,K)$. Then
$$\Gamma_2=(-\Sigma,\mbox{\boldmath$\beta$},\mbox{\boldmath${\alpha}$},
z,w)$$ is also a Heegaard diagram for $(Y,K)$. If $\mathbf
x\in\mathbb T_{\alpha}\cap\mathbb T_{\beta}$ represents $\mathfrak
s$ in $\Gamma_1$, then $\mathbf x$ represents $J\mathfrak s$ in
$\Gamma_2$. If $\phi$ is a holomorphic disk in $\Gamma_1$
connecting $\mathbf x$ to $\mathbf y$, then $\phi$ gives rise to a
holomorphic disk $\overline{\phi}$ in $\Gamma_2$ connecting
$\mathbf x$ to $\mathbf y$. Topologically, $\overline{\phi}$ is
just $-\phi$. Let $\omega\subset \Sigma$ be a curve representing
$\zeta$, then
\begin{equation}\label{eq:-omega}
\omega\cdot\partial_{\alpha}\phi|_{\Sigma}=-\omega\cdot\partial_{\beta}\phi|_{\Sigma}
=-\omega\cdot\partial_{\beta}\overline{\phi}|_{-\Sigma},
\end{equation}
 where
the notation $|_{\Sigma}$ or $|_{-\Sigma}$ implies that the
intersection number is evaluated in $\Sigma$ or $-\Sigma$.

Since $\mathfrak s$ is torsion, there is a unique
$\xi_0\in\underline{\mathrm{Spin}^c}(Y,K)$ extending $\mathfrak s$
which satisfies
$$\langle c_1(\xi_0),[\widehat F]\rangle=0$$
for any Seifert surface $F$ for $K$. Using the observations in the
first paragraph, it follows that if we interchange the roles of
$i$ and $j$, then the chain complex $CFK^{\infty}(Y,K,\xi_0)$ can
be viewed as the chain complex $CFK^{\infty}(Y,K,J\xi_0)$. It
follows that there is a grading preserving isomorphism
$$CFK^{-m,0}(Y,K,\xi_0)\cong
CFK^{0,-m}(Y,K,J\xi_0)\cong\widehat{CFK}(Y,K,J\xi).
$$
Moreover, the map $U^m$ induces an isomorphism
$$U^m\co \widehat{CFK}(Y,K,\xi)=CFK^{0,m}(Y,K,\xi_0)\to CFK^{-m,0}(Y,K,\xi_0)$$
which decreases the absolute grading by $2m$. Hence
$$\widehat{CFK}_d(Y,K,\xi)\cong\widehat{CFK}_{d-2m}(Y,K,J\xi).$$

Using (\ref{eq:-omega}), we find that the action of $A_{\omega}$
on $\widehat{CFK}_d(Y,K,\xi)$ corresponds to the action of
$A_{-\omega}$ on $\widehat{CFK}_{d-2m}(Y,K,J\xi)$.
\end{proof}

\section{The contact invariant $EH(M,\gamma,\xi)$}

Suppose $(M,\gamma)$ is a sutured manifold. A contact structure on $(M,\gamma)$ is a contact structure $\xi$ on $M$, such that $\partial M$ is convex and the suture $s(\gamma)$ is the dividing set. Suppose 
$$(M,\gamma)\stackrel{S}{\rightsquigarrow}(M',\gamma')$$ is a taut decomposition, and $\xi$ is a contact structure on $(M,\gamma)$ such that $S$ is convex and the dividing set $\gamma_S$ on $S$ is $\partial$--parallel, namely, each component of $\gamma_S$ cuts off a disk containing no other component of $\gamma_S$. Let $\xi'$ be the restriction of $\xi$ on $(M',\gamma')$. Then $\xi'$ is tight if and only if $\xi$ is tight \cite{HKM0}.

\begin{defn}
Suppose $(M,\gamma)$ is a taut sutured manifold. A tight contact structure $\xi$ on $(M,\gamma)$ is of {\it hierarchy type}, if there exists a well-groomed sutured manifold hierarchy
\begin{equation}\label{eq:SutHier}
(M,\gamma)=(M_0,\gamma_0)\stackrel{S_0}{\rightsquigarrow}(M_1,\gamma_1)\stackrel{S_1}{\rightsquigarrow}(M_2,\gamma_2)\stackrel{S_2}{\rightsquigarrow}\cdots\stackrel{S_{n-1}}{\rightsquigarrow}(M_n,\gamma_{n}),
\end{equation}
such that the dividing set on each $S_i$ is $\partial$--parallel. In fact, since $M_n$ consists of balls, $\xi$ is obtained by gluing the unique tight contact structure on $M_n$ along the decomposing surfaces.
\end{defn}

For a contact structure $\xi$ on $(M,\gamma)$, Honda, Kazez and Mati\'c \cite{HKM} defined an invariant $EH(M,\gamma,\xi)\in SFH(-M,-\gamma)/(\pm1)$. They also studied the behavior of this invariant under 
sutured manifold decomposition.

\begin{thm}[Honda--Kazez--Mati\'c]\label{thm:HKM}
Let $(M,\gamma,\xi)$ be the contact structure obtained from $(M',\gamma',\xi')$ by gluing along a $\partial$--parallel $(S,\gamma_S)$. 
Under the inclusion of $SFH(-M',-\gamma')$ into $SFH(-M,-\gamma)$ as a direct summand, $EH(M',\gamma',\xi')$ is mapped to $EH(M,\gamma,\xi)$.
\end{thm}

Now Theorem~\ref{thm:Decomp} immediately implies the following result.

\begin{cor}\label{cor:EH}
Suppose $(M,\gamma)$ is a taut sutured manifold and $\xi$ is a contact structure of hierarchy type on $M$. Then for any $\zeta\in H_1(M,\partial M;\mathbb Z)/\mathrm{Tors}$, $EH(M,\gamma,\xi)$ lies in the kernel of $A_{\zeta}$ 
while $EH(M,\gamma,\xi)$ is not contained in the image of $A_{\zeta}$.
\end{cor}
\begin{proof}
Consider a hierarchy (\ref{eq:SutHier}) associated to the contact structure $\xi$. By Theorem~\ref{thm:HKM}, there are maps
$$\iota\co SFH(-M_n,-\gamma_n)\cong\mathbb Z\to SFH(-M,-\gamma),$$
which sends a generator of $SFH(-M_n,-\gamma_n)$ to $EH(M,\gamma,\xi)$, and
$$\pi\co SFH(-M,-\gamma)\to SFH(-M_n,-\gamma_n)\cong\mathbb Z,$$
which sends $EH(M,\gamma,\xi)$ to a generator of $SFH(-M_n,-\gamma_n)$.

Given $\zeta\in H_1(M,\partial M)$, let $i_*(\zeta)$ be its image in $H_1(M_n,\partial M_n)$. Since $M_n$ consists of balls, $i_*(\zeta)=0$, hence $A_{i_*(\zeta)}=0$. (This result also follows from the fact that $A_{i_*(\zeta)}$ is a differential.) 
Using Theorem~\ref{thm:Decomp}, we get a commutative diagram:
$$
\begin{CD}
{SFH}(-M_n,-\gamma_n)@>\iota>>{SFH}(-M,-\gamma)\\
@V0VV@VA_{\zeta}VV\\
{SFH}(-M_n,-\gamma_n)@>\iota>>{SFH}(-M,-\gamma),
\end{CD}
$$
hence $A_{\zeta}(EH(M,\gamma,\xi))=0$.

Similarly, considering the commutative diagram 
$$
\begin{CD}
{SFH}(-M,-\gamma)@>\pi>>{SFH}(-M_n,-\gamma_n)\\
@VA_{\zeta}VV@V0VV\\
{SFH}(-M,-\gamma)@>\pi>>{SFH}(-M_n,-\gamma_n),
\end{CD}
$$
we conclude that $EH(M,\gamma,\xi)$ does not lie in the image of $A_{\zeta}$.
\end{proof}

\begin{rem}
Since $A^2_{\zeta}=0$, 
$A_{\zeta}$ can be viewed as a differential on the Floer homology group, thus one can talk about its homology. 
Corollary~\ref{cor:EH} says that the contact invariant represents a nontrivial class in the homology of $A_{\zeta}$. 
A version of Corollary~\ref{cor:EH} for weakly fillable contact structures on closed manifolds was proved in \cite{HN}, 
following the strategy of Ozsv\'ath and Szab\'o \cite{OSzGenus}. 
\end{rem}

\begin{proof}[Proof sketch of Corollary~\ref{cor:KerFiber}]
Decomposing $Y-K$ along $F$, we get a taut sutured manifold $(M,\gamma)$. If $F$ is not a fiber of any fibration, then $(M,\gamma)$ is not a product sutured manifold.
The argument in \cite{Gh,NiFibred} shows that we can construct two different sutured manifold hierarchies, and corresponding two tight contact structures $\xi_1,\xi_2$ obtained 
from gluing the tight contact structure on the balls via the two hierarchies, such that $EH(M,\gamma,\xi_1)$ and $EH(M,\gamma,\xi_2)$ are linearly independent. 
See also \cite{Ju2} for the version of argument adapted to sutured Floer homology.

It is showed in \cite{Ju2} that the inclusion map 
$$SFH(M,\gamma)\to HFK(Y,K)$$
induced by the decomposition
$$Y-K\stackrel{F}{\rightsquigarrow}(M,\gamma)$$
maps $SFH(M,\gamma)$ isomorphically onto $HFK(Y,K,[F],-g)$. Using Theorem~\ref{thm:Decomp} and Corollary~\ref{cor:EH}, 
we conclude that $\mathrm{Ker}A$ has rank at least $2$.
\end{proof}

The reader may find that the use of the contact class $EH$ is not necessary for the proof of Corollary~\ref{cor:KerFiber}. 
We choose this presentation so that the nontrivial elements in $\mathrm{Ker}A$ have their geometric meaning.

\section{Knots in $\#^{n}S^1\times S^2$}

Let $Y_n=\#^{n}S^1\times S^2$, $V_n=H_1(Y_n;\mathbb Z)$, $V_n'=H^1(Y_n;\mathbb Z)$. It is
well known that $\widehat{HF}(Y_n)$ as a $\Lambda^*V_n$--module is
isomorphic to $\Lambda^*V'_n$. Namely,
$\widehat{HF}(Y_n)\cong\Lambda^*V_n'$ as a group, and $A_{\zeta}$
is given by the contraction homomorphism
$$\iota_{\zeta}\co \Lambda^iV'_n\to \Lambda^{i-1}V'_n.$$

\begin{lem}\label{lem:Ker1}
In the module $\Lambda^*V_n'$, we have
$$\bigcap_{\zeta\in V_n}\mathrm{ker}\:\iota_{\zeta}=\mathbb Z\mathbf 1,$$
the subgroup generated by the unit element $\mathbf1$.
\end{lem}
\begin{proof}
Clearly $\mathbf1$ is in the kernel of all $\iota_{\zeta}$. Suppose $x\in\Lambda^*V'_n$, and the highest degree summand of $x$ has degree $i>0$. Let $\{\zeta_1,\dots,\zeta_n\}$ be a set of generators of
$V_n$, and 
let $\{\zeta_1',\dots,\zeta_n'\}$ be a basis of $V_n'$ such that $\zeta_i'(\zeta_j)=\delta_{ij}$. 
Without loss of generality, we can assume $x$ contains a term $k\zeta_1'\wedge\zeta_2'\wedge\cdots\wedge\zeta_i'$ where $k\in\mathbb Z\backslash\{0\}$, then 
$$\iota_{\zeta_1}(x)=k\:\zeta_2'\wedge\cdots\wedge\zeta_i'+\cdots$$
is nonzero. This finishes our proof.
\end{proof}

\begin{prop}\label{prop:FiberGenus}
Suppose $K\subset Y_{n}$ is a knot with simple Floer homology, and $Y_n-K$ is irreducible, then the genus of $K$ is $g=\frac n2$ and $K$ is
fibered.
\end{prop}
\begin{proof}
Let $F$ be a minimal genus Seifert surface of genus $g$. By definition $$A_{\zeta}\co\widehat{CF}(Y_n)\to\widehat{CF}(Y_n)$$ is a filtered map, so $\widehat{CFK}(Y_n,K,[F],-g)$ is a subcomplex of $\widehat{CF}(Y_n)$. 
Moreover, since $K$ has simple Floer homology group, the rank of $\widehat{HFK}(Y_n,K)$ is the same as the rank of $\widehat{HF}(Y_n)$, 
so $\widehat{HFK}(Y_n,K,[F],-g)$ is a submodule of the $\Lambda^*V_n$--module $\widehat{HF}(Y_n)$. Similarly, $\widehat{HFK}(Y_n,K,[F],g)$ is a quotient module of $\widehat{HF}(Y_n)$.
 
Let $\mathrm{Ker}A$ be the subgroup of $\widehat{HFK}(Y_n,K,[F],-g)$ which is the intersection of $\mathrm{ker}A_{\zeta}$ for all $\zeta\in V_n$, then Lemma~\ref{lem:Ker1} shows that the rank of $\mathrm{Ker}A$ is at most $1$.
It follows from Corollary~\ref{cor:KerFiber} that $K$ is a fibered knot with fiber $F$, and $\widehat{HFK}(Y_n,K,[F],-g)$ is generated by $\mathbf1$.

Any
monomial $\zeta'_{i_1}\wedge\cdots\wedge\zeta'_{i_k}\in\widehat{HF}(Y_n)$ can be obtained by applying
a series of $A_{\zeta}$'s to $\Delta=\zeta'_1\wedge\cdots\wedge\zeta'_n$. Since $A_{\zeta}$ is a filtered map, we see that
$\Delta$ has the highest filtration, hence $\widehat{HFK}(Y_n,K,[F],g)$ is generated by the image of $\Delta$ under the quotient map $\widehat{HF}(Y_n)\to\widehat{HFK}(Y_n,K,[F],g)$.

By \cite[Proposition~3.10]{OSzKnot}, the difference between the Maslov grading of $\Delta$ and the Maslov grading of $\mathbf1$ is $2g$. 
On the other hand, since $$\mathbf1=A_{\zeta_1}\circ\cdots\circ A_{\zeta_n}(\Delta),$$ the difference between the Maslov grading of $\Delta$ and the Maslov grading of $\mathbf1$ is $n$.
So $n=2g$.
\end{proof}

\begin{proof}[Proof of Theorem~\ref{thm:Borromean}]
If $Y_n-K$ is reducible, then it has a $S^1\times S^2$ connected summand. We can remove this summand and regard $K$ as a knot in $Y_{n-1}$, which still has simple Floer homology group. 
Hence we may assume that $Y_n-K$ is irreducible.
By Proposition~\ref{prop:FiberGenus}, $n=2g$ where $g$ is the genus of $K$, and $K$ is fibered. Let
$F$ be the Seifert surface of $K$ which is a fiber of the
fibration. Pick a base point on $\partial F$. Let $\varphi\co F\to
F$ be the monodromy of the fibration such that $\varphi|_{\partial
F}$ is the identity. Let $\varphi_*\co\pi_1(F)\to \pi_1(F)$ be the induced map on $\pi_1(F)$. Let
$t$ represent a meridian of $K$, then
$$\pi_1(Y_{2g}-K)=\langle \pi_1(F),t\:|\:t\varphi_*(a)t^{-1}a^{-1}=1,\:\forall a\in\pi_1(F)\rangle.$$
After filling $Y_{2g}-K$ along the meridian, $t$ is killed, so we
get
\begin{equation}\label{eq:Quotient}
\pi_1(Y_{2g})=\langle \pi_1(F)\:|\:\varphi_*(a)a^{-1}=1,\:\forall
a\in\pi_1(F)\rangle,
\end{equation}
 which is a quotient group of $\pi_1(F)$. We
know that $\pi_1(F)$ is a free group of rank $2g$. On the other
hand, $Y_{2g}$ is the connected sum of $2g$ copies of $S^1\times
S^2$, so its $\pi_1$ is also a free group of rank $2g$. Since free
groups of finite ranks are Hopfian \cite{MKS}, the
relations in (\ref{eq:Quotient}) are all trivial, hence
$\varphi_*=\mathrm{id}$.

Now it is a standard fact that $\varphi$ is isotopic to the
identity map on $F$ through maps which fix $\partial F$ pointwise.
In fact, we define a map $$\Phi\co (F\times\{0,1\})\cup(\partial
F\times[0,1])\to F$$ by letting
$$\Phi(x,0)=\varphi(x), \Phi(x,1)=x,\quad\forall x\in F$$ and $$\Phi(x,t)=x,\quad \forall x\in\partial F.$$
Since $\varphi_*=\mathrm{id}$ and $F$ is a $K(\pi,1)$, we can
extend $\Phi$ to a map from $F\times I$ to $F$. This means that
$\varphi$ is homotopic hence isotopic to the identity map relative
to $\partial F$.

Since the monodromy $\varphi$ is isotopic to the identity, the
complement of $K$ is homeomorphic to $F\times S^1$, which is
homeomorphic to the complement of $B_g$. By a result of Gabai
\cite[Corollary~2.14]{G2}, knots in $Y_{2g}$ which do not lie in a
$3$--cell are determined by their complements. So $K=B_g$.
\end{proof}

\section{Links in $S^3$}

In this section, we will study links in $S^3$ that have simple Floer homology groups.
Ozsv\'ath and Szab\'o \cite{OSzLink} defined a multi-graded
$\mathbb Z/2\mathbb Z$--coefficient homology theory for links,
called link Floer homology, denoted $\widehat{HFL}(\cdot)$.
Although link Floer homology generally contains more information
than the knot Floer homology of a link, the rank of
$\widehat{HFL}(L)$ is equal to the rank of $\widehat{HFK}(L)$
\cite[Theorem~1.1]{OSzLink}. So Proposition~\ref{prop:LinkRank} can also be
stated in terms of link Floer homology. In fact, we will mainly
work with link Floer homology in our proof.

\begin{proof}[Proof of Proposition~\ref{prop:LinkRank}]
Without loss of generality, we will work with $\mathbb F=\mathbb
Z/2\mathbb Z$ coefficients. We will induct on the number of
components of $L$.

When $n=|L|=1$, the result is a consequence of Ozsv\'ath--Szab\'o
\cite{OSzGenus}. Assume that our result is already proved for
$(n-1)$--component links and let $L$ be an $n$--component link
such that
$$\mathrm{rank}\:\widehat{HFK}(L)=\mathrm{rank}\:\widehat{HFL}(L)=2^{n-1}.$$

If $L$ has a trivial component which bounds a disk in the
complement of $L$, then we can remove this component and apply the
induction hypothesis to conclude that $L$ is the unlink. From now
on, we assume $L$ has no trivial component.

Let $K_1$ be a component of $L$. Let $M$ be the rank two graded
vector space with one generator in grading $0$ and another in
grading $-1$. By Ozsv\'ath--Szab\'o
\cite[Proposition~7.1]{OSzLink}, there is a differential $D^1$ on
$\widehat{HFL}(L)$, such that the homology of
$(\widehat{HFL}(L),D^1)$ is $\widehat{HFL}(L-K_1)\otimes M$. Here
the two Floer homology groups have Alexander gradings in
$\underline{\mathrm{Spin}^c}(L-K_1)$, and the isomorphism is up to
some overall translation of the gradings. So the rank of
$\widehat{HFL}(L-K_1)$ is less than or equal to $2^{n-2}$.

Since $L-K_1$ is an $(n-1)$--component link, the rank of its link
(knot) Floer homology is greater than or equal to $2^{n-2}$, so
the rank should be exactly $2^{n-2}$. Hence the differential
$D^1=0$ and
$$\widehat{HFL}(L)\cong\widehat{HFL}(L-K_1)\otimes M$$
up to an overall translation of the gradings, where the Alexander
gradings are in $\underline{\mathrm{Spin}^c}(L-K_1)$. By the
induction hypothesis, $L-K_1$ is the $(n-1)$--component unlink,
hence its $\widehat{HFL}$ is supported in exactly one Alexander
grading. It follows that $\widehat{HFL}(L)$ is supported in
exactly one element in $\underline{\mathrm{Spin}^c}(L-K_1)$, thus
$\widehat{HFL}(L)$ is supported in one line in
$\underline{\mathrm{Spin}^c}(L)$. Now Ozsv\'ath--Szab\'o
\cite[Theorem~1.1]{OSzLinkNorm} implies that there exists a
nonzero element $h\in H^1(S^3-L;\mathbb Z)$, such that
$$x(\mathrm{PD}[h])+\sum_{i=1}^{n}|\langle h,\mu_i\rangle|=0,$$
where $\mu_i$ is the meridian of the $i$-th component of $L$. Thus
$|\langle h,\mu_i\rangle|=0$ for each $i$, which is impossible
since $h\ne0$, a contradiction.
\end{proof}

\end{document}